\documentclass[12pt]{amsart}

\usepackage{amsmath, amssymb, amsthm}

\title[Thick Difference Sets]{Thick Difference Sets of Haar Null Compact Sets in Locally Compact Groups}

\author{Chuck Akemann}
\address{Department of Mathematics, University of California, Santa Barbara, CA 93106, USA}
\email{akemann@math.ucsb.edu}


\thanks{The author thanks AI models Super Grok, ChatGPT, and Gemini for very helpful discussions that contributed substantially to this work.  In particular, these models {\bf cooperatively} produced clean TeX output and found literature references at almost unbelievable speed.}

\newtheorem{theorem}{Theorem}
\newtheorem{lemma}{Lemma}

\begin{document}

\begin{abstract}
Let \(G\) be a non-discrete, locally compact group with Haar measure \(m\).
We prove that there exists a compact set \(K \subset G\) with \(m(K)=0\)
such that \(KK^{-1}\) contains a neighborhood of the identity.
Moreover, such a set may be constructed inside any prescribed
neighborhood of the identity.
\end{abstract}

\maketitle

\section{Introduction}

Let \(G\) be a non-discrete, locally compact group with left Haar measure \(m\).
The classical Steinhaus--Weil theorem asserts that if \(A \subset G\)
has positive measure, then \(AA^{-1}\) contains a neighborhood of the identity
(see, e.g., Hewitt and Ross \cite[Corollary~20.17]{HR}).

That the converse fails is well known in abelian groups (for example on the torus),
but the general nonabelian case does not appear to be explicitly recorded.

We prove that the converse fails in a strong form for all non-discrete locally compact groups.

\section{Main result}

\begin{theorem}\label{thm:main}
Let \(G\) be a non-discrete, locally compact group with Haar measure \(m\),
and let \(O\) be an open neighborhood of the identity.
Then there exists a compact set \(K \subset O\) such that
\[
m(K)=0 \quad \text{and} \quad KK^{-1}
\text{ contains a neighborhood of the identity.}
\]
\end{theorem}

\begin{lemma}[Pullback of null sets]
Let \(\rho : G \to H\) be a continuous open surjective homomorphism of locally compact groups
with compact kernel. If \(K_H \subset H\) is compact with Haar measure zero in \(H\),
then \(K := \rho^{-1}(K_H)\) is compact with Haar measure zero in \(G\).
\end{lemma}

\begin{proof}
Compactness of \(K\) follows because the fibers are compact and \(K_H\) is compact.
By Weil's integration formula (see, e.g., Hewitt and Ross \cite[\S 28]{HR}),
\[
m_G(K) = \int_H m_N(K \cap \rho^{-1}(h))\,dm_H(h).
\]
Because \(K = \rho^{-1}(K_H)\), the intersection \(K \cap \rho^{-1}(h)\) is the entire fiber \(\rho^{-1}(h)\) (which has measure \(m_N(\ker\rho)\)) when \(h \in K_H\), and is empty otherwise. The integrand therefore vanishes whenever \(h \notin K_H\), and since \(m_H(K_H)=0\) we conclude \(m_G(K)=0\).
\end{proof}

\section{Lie case}

\begin{lemma}
Let \(H\) be a Lie group of positive dimension and let \(W\)
be a neighborhood of the identity. Then there exists a compact set
\(K_H \subset W\) with \(m_H(K_H)=0\) such that \(K_HK_H^{-1}\)
contains a neighborhood of the identity.
\end{lemma}

\begin{proof}
Let \(\mathfrak{h}\) be the Lie algebra of \(H\) and fix a decomposition
\[
\mathfrak{h} = V \oplus \mathbb{R} X_1.
\]
Choose \(\delta>0\) sufficiently small so that the map
\[
\Theta(v,t) = \exp(v)\exp(tX_1)
\]
is a diffeomorphism from
\[
B_V(0,3\delta) \times (-3\delta,3\delta)
\]
onto an open neighborhood \(U \subset W\) of the identity (see, e.g., Hall \cite[Ch.~2]{Hall}).

Let \(C \subset [-\delta,\delta]\) be a compact null set (e.g., a scaled ternary Cantor set) such that
\[
C - C \supset [-\delta,\delta].
\]
Set
\[
E := \overline{B_V(0,\delta)} \times C, \quad K_H := \Theta(E).
\]
Since \(E\) has Lebesgue measure zero in \(V \times \mathbb{R}\) and \(\Theta\) is smooth (hence locally Lipschitz),
\(m_H(K_H)=0\).

For \(v \in \overline{B_V(0,\delta)}\) and \(t,s \in C\),
\[
\Theta(v,t)\Theta(0,s)^{-1}
= \exp(v)\exp((t-s)X_1).
\]
Thus \(K_HK_H^{-1}\) contains
\[
\Theta(\overline{B_V(0,\delta)} \times [-\delta,\delta]),
\]
whose interior contains a neighborhood of the identity in \(H\).
\end{proof}

\section{Profinite case}

\begin{lemma}[Finite difference bases]
There exists an absolute constant \(A = 4/\sqrt{3}\) such that for every finite group \(F\)
there exists \(T \subset F\) with \(TT^{-1} = F\) and
\[
|T| \le A \sqrt{|F|}.
\]
\end{lemma}

\begin{proof}
This is the main result of Kozma and Lev \cite[Theorem~1]{KL}.
\end{proof}

\begin{lemma}
Let \(G\) be an infinite compact totally disconnected group.
Let \(O\) be a neighborhood of the identity.
Then there exists a compact set \(K \subset O\) with \(m(K)=0\)
such that \(KK^{-1}\) contains a neighborhood of the identity.
\end{lemma}

\begin{proof}
Since \(G\) is compact and totally disconnected,
it is profinite and admits a neighborhood basis at the identity
consisting of open normal subgroups (see, e.g., Hofmann and Morris \cite[Theorem~1.34]{HM}).

We construct a descending sequence of open normal subgroups
\[
N_1 \supset N_2 \supset \cdots,\quad N_1 \subset O
\]
inductively so that the successive quotients \(Q_k := N_k/N_{k+1}\) satisfy
\[
\frac{A}{\sqrt{|Q_k|}} \le \frac12
\]
(where \(A=4/\sqrt{3}\)). This is possible because \(G\) is infinite profinite:
given \(N_k\), the open normal subgroups contained in \(N_k\) still form a neighborhood basis of the identity, so we can skip intermediate subgroups to ensure the index \([N_k:N_{k+1}]=|Q_k|\) is arbitrarily large.

Let \(H_k = N_1/N_{k+1}\) and let \(\Phi_k : N_1 \to H_k\) be the canonical quotient map.
Let \(\pi_k : H_{k+1} \to H_k\) be the natural projection between quotients, so that \(\Phi_k = \pi_k \circ \Phi_{k+1}\).

We construct subsets \(S_k \subset H_k\) recursively so that
\[
S_k S_k^{-1} = H_k, \quad \pi_k(S_{k+1}) = S_k,
\]
and writing
\[
\eta_k := \frac{|S_k|}{|H_k|},
\]
we have
\[
\eta_{k+1} \le \eta_k \frac{A}{\sqrt{|Q_{k+1}|}} \le \frac{\eta_k}{2}.
\]

First choose \(S_1 \subset H_1\) such that
\[
S_1 S_1^{-1} = H_1 \quad\text{and}\quad |S_1| \le A \sqrt{|H_1|}.
\]

Suppose \(S_k\) has been constructed. Choose
\(T_{k+1} \subset Q_{k+1}\) such that
\[
T_{k+1} T_{k+1}^{-1} = Q_{k+1} \quad\text{and}\quad |T_{k+1}| \le A \sqrt{|Q_{k+1}|}.
\]

For each \(x \in S_k\), choose a lift \(\tilde{x} \in H_{k+1}\)
with \(\pi_k(\tilde{x}) = x\), and let \(\tilde{S}_k\) denote
the set of these lifts. Define
\[
S_{k+1} := \tilde{S}_k T_{k+1}.
\]

Because \(T_{k+1} T_{k+1}^{-1} = Q_{k+1}\) and \(Q_{k+1}\) is normal in \(H_{k+1}\), we have
\[
S_{k+1} S_{k+1}^{-1} = \tilde{S}_k T_{k+1} T_{k+1}^{-1} \tilde{S}_k^{-1} = \tilde{S}_k Q_{k+1} \tilde{S}_k^{-1} = \tilde{S}_k \tilde{S}_k^{-1} Q_{k+1}.
\]
Since \(\pi_k(\tilde{S}_k \tilde{S}_k^{-1}) = S_k S_k^{-1} = H_k\), the product covers all of \(H_k\), and multiplying by the kernel \(Q_{k+1}\) yields all of \(H_{k+1}\). Thus \(S_{k+1} S_{k+1}^{-1} = H_{k+1}\).

Moreover,
\[
\eta_{k+1} = \frac{|S_{k+1}|}{|H_{k+1}|} \le \frac{|S_k||T_{k+1}|}{|H_k||Q_{k+1}|} \le \eta_k \frac{A}{\sqrt{|Q_{k+1}|}} \le \frac{\eta_k}{2}.
\]

Thus \(\eta_k \le \eta_1 / 2^{k-1} \to 0\).

Define
\[
K_k := \Phi_k^{-1}(S_k) \subset N_1.
\]
Because \(\Phi_k = \pi_k \circ \Phi_{k+1}\) and \(\pi_k(S_{k+1}) = S_k\), it follows that \(K_{k+1} \subset K_k\). Each \(K_k\) is compact, and
\[
m(K_k) = \eta_k\, m(N_1) \to 0.
\]

Let
\[
K = \bigcap_{k=1}^\infty K_k.
\]
Then \(K\) is compact, \(K \subset N_1 \subset O\), and \(m(K)=0\).

Let \(g \in N_1\). For each \(k\), since \(H_k = S_k S_k^{-1}\),
there exist \(a_k, b_k \in S_k\) such that
\[
\Phi_k(g) = a_k b_k^{-1}.
\]
Thus \(K_k \cap gK_k \neq \emptyset\) for all \(k\).

Since these sets are nested and compact, their intersection is nonempty,
so \(g \in KK^{-1}\).

Hence \(N_1 \subset KK^{-1}\).
\end{proof}

\section{Proof of the main theorem}

\begin{proof}
Let \(G\) be a non-discrete, locally compact group and let \(O\)
be an open neighborhood of the identity.

By Yamabe's theorem (see, e.g., Montgomery and Zippin \cite[Section~4.6]{MZ}),
there exists an open subgroup \(G'\) of \(G\) and a compact normal subgroup \(N\subset O\cap G'\) of \(G'\)
such that \(H = G'/N\) is a Lie group. Let \(\rho : G' \to H\) be the canonical quotient map.

\textbf{Case 1: \(\dim H > 0\).}
Choose an open neighborhood \(W\) of the identity in \(H\) such that \(\rho^{-1}(W) \subset O \cap G'\).
By the Lie case, there exists a compact set \(K_H \subset W\) with \(m_H(K_H)=0\) such that \(K_H K_H^{-1}\) contains a neighborhood \(V\) of the identity in \(H\).
Set \(K := \rho^{-1}(K_H)\). Then \(K\) is compact and \(K \subset O\).
By the pullback lemma, \(m_G(K)=0\).
Moreover, \(KK^{-1} = \rho^{-1}(K_H K_H^{-1}) \supset \rho^{-1}(V)\), which is open in \(G'\) and hence in \(G\).

\textbf{Case 2: \(\dim H = 0\).}
Then \(H\) is discrete, so \(N\) is open in \(G'\) (hence open in \(G\)) and compact.
Since \(G\) is non-discrete, \(N\) is infinite.
By the structure theory of compact groups (Hofmann–Morris \cite[Chapter 9]{HM}), either \(N\) is totally disconnected or there exists a closed normal subgroup \(M \subset N\) such that \(N/M\) is a compact Lie group of positive dimension.

- If \(N\) is totally disconnected, apply the profinite-case lemma inside \(N\) to obtain the desired \(K \subset N \subset O\).
  Since \(N\) is open, any neighborhood in \(N\) is a neighborhood in \(G\).
- If \(N/M\) has positive dimension, repeat Case 1 inside \(N\): the resulting set \(K \subset N\) is compact, null in \(N\) (hence in \(G\)), and \(KK^{-1}\) contains a neighborhood of the identity in \(N\) (hence in \(G\)).

This completes the proof.
\end{proof}

\end{document}